\documentclass[review,onefignum,onetabnum]{siamart190516}
\usepackage{braket,amsfonts,amsopn}
\usepackage{graphicx,epstopdf}
\usepackage[caption=false]{subfig}
\usepackage{array} 
\usepackage[caption=false]{subfig}

\usepackage{diagbox}

\usepackage{mathrsfs} 
 
\usepackage{booktabs} 

\usepackage{algorithmic} 
\Crefname{ALC@unique}{Line}{Lines} 

\title{	Convergence estimation and characteristic analysis of a two-level iterative algorithm for the discretized three-temperature energy linear systems 
}
\author{ Yue Hao   \thanks{Laboratory of Computational Physics, Institute of Applied Physics and Computational Mathematics, Beijing 100088, China.   Email addresses: \email{hao\_yue1993@163.com}(Yue Hao),  \email{huangsilu1992@163.com}(Silu Huang), \email{xwxu@iapcm.ac.cn}(Xiaowen Xu).}
\and  Silu  Huang \footnotemark[1]
\and  Xiaowen Xu \footnote{Corresponding author}  \footnotemark[1] 
}

\begin{document}

\maketitle

\begin{abstract}

  For solving the discretized three-temperature energy linear systems, Xu et al. proposed a physical-variable based coarsening two-level iterative method (PCTL algorithm) in 2009 and verified its efficiency by numerical experiments in practical applications. In this paper, we study in detail the specific convergence property of the PCTL algorithm based on the theory of algebraic multigrid method (AMG), and give a reasonable upper bound on the convergence factor, which provides a theoretical guarantee for the PCTL algorithm.
  Moreover, we also analyse the algebraic features that affect the convergence of the PCTL algorithm, such as diagonal dominance and coupling strength, hoping provides theoretical guidance for the applications and algorithm optimization of the PCTL algorithm.

\end{abstract}
  
   

\section{Introduction}\label{sec:intro}
In this paper, we aim to evaluate the convergence speed of the physical-variable based coarsening two level (PCTL) iterative algorithm \cite{Xu2009algebraic,Huang2022aSetup}, which is proposed
for solving the following linear systems arising from the three-temperature (3-T) energy equation in radiation hydrodynamics (RHD) applications:
\begin{equation}\label{eq:main}
Ax:= \begin{bmatrix} 
  A_r &  0   & D_{re} \\
  0   & A_i  & D_{ie} \\
  D_{er} & D_{ei}  & A_e 
\end{bmatrix} 
\begin{bmatrix}  x_r \\  x_i \\ x_e  \end{bmatrix} 
= \begin{bmatrix}  b_r \\  b_i \\ b_e  \end{bmatrix} =:b,
\end{equation}
where $A \in \mathbb{R}^{3n \times 3n}$ with $n$ be the number of grids, the diagonal submatrices $A_{\alpha} = (a_{kj}^{\alpha})_{n \times n} \in \mathbb{R}^{n \times n} \, (\alpha = r, i, e)$ reflect the diffusion of the radiation, ion and electron temperatures, and the coupling terms $D_{\alpha \beta} \in \mathbb{R}^{n \times n} \, (\alpha, \beta = r, i, e)$ express the energy exchange between the $\alpha$-th and the $\beta$-th physical quantities. In general, the 3-T linear system has the following properties: 
\begin{itemize}
  \item[(1)] The diagonal submatrices $A_{\alpha} \, (\alpha = r, i, e)$ are {\it M-matrix}, that is $a_{kk}^{\alpha} >0$, $a_{kj}^{\alpha} \leq 0 \, (j \neq k)$, and all elements of $A_{\alpha}^{-1}$ are nonnegative. Moreover, $A_{\alpha}$ is also a symmetric and positive definite matrix.
  \item[(2)] The coupling terms $D_{\alpha \beta} \, (\alpha, \beta = r, i, e)$
  are {\it diagonal matrices} with {\it negative }diagonal elements.
  \item[(3)] The coefficient matrix $A=(a_{kj})$ is {\it strong diagonally dominant} with $a_{kk} > \sum_{j \neq k} |a_{kj}|$ for all $1 \leq k \leq 3n$.
\end{itemize}
In this paper, our analysis is focused on the symmetric case with $D_{re}=D_{er}^T$ and $D_{ie}=D_{ei}^T$.  

Solving 3-T energy equations is an important task of the numerical simulation of RHD problems, which arising from many fields, such as the inertial confinement fusion (ICF), astrophysical phenomena and so on \cite{Mihalas1984foundations,Pomraning1973the}. However, complex application features make the coefficient matrix in 3-T linear system \cref{eq:main} generally ill-conditioned and difficult to solve, such that solving 3-T linear systems takes up most of the time of RHD simulation. Thus, developing an efficient and practical algorithm for solving \cref{eq:main} is a crucial problem. 
In recent years, numerous methods have been proposed for solving 3-T linear systems, in which preconditioned Krylov subspace methods \cite{Saad2003iterative,An2019operator} are the most favorable choices,  with the   
preconditioners mainly include incomplete LU factorization \cite{Baldwin1999iterative}, geometric and algebraic multigrid \cite{Baldwin1999iterative,Mo2004parallel,Xu2009algebraic,Xu2017algebraic,Yue2020algebraic,Huang2022aSetup}, domain decomposition \cite{Hu2017domain,Yue2018substructuring} and their effective combination.


 In \cite{Xu2009algebraic}, a physical-variable based coarsening two-level (PCTL) iterative algorithm was proposed for solving the 3-T linear system \cref{eq:main}.
 Based on a specific coarsen strategy, the PCTL algorithm divides the fully coupled system into four individual easier-to-solve subsystems, and thus addresses the difficulties caused by the complicated couplings among physical quantities.
 Then, Zhou et al. \cite{Zhou2012an} tested the efficiency of the PCTL algorithm in practical applications and discussed the impact factor based on the numerical results.
 Recently, in order to further improve the efficiency of the PCTL algorithm for solving sequences of 3-T linear systems with dynamically and slowly changing features, Huang et al. proposed an $\alpha$Setup-PCTL algorithm \cite{Huang2022aSetup} by adaptively selecting the appropriate solution strategies for each linear system. These algorithms have been integrated into the parallel algebraic multigrid solver (JXPAMG) \cite{Xu2022jxpamg} and widely used in the simulation of practical applications. 
 Numerical results \cite{Xu2009algebraic,Zhou2012an,Huang2022aSetup} have shown the high efficiency and scalability of these PCTL-like algorithms both as solvers and as preconditioners for Krylov subspace methods when applied to the 3-T linear systems. 
 

 Note that the convergence and the efficiency of the PCTL algorithm are only observed from the numerical experiments and have not been theoretically analyzed yet. Actually, the PCTL method is a kind of algebraic two-grid (ATG) method, and thus general frameworks for analyzing the convergence of the ATG method \cite{Ruge1987algebraic,Stuben1999algebraic,Trottenberg2000multigrid,Falgout2004on,Falgout2005on,Notay2007convergence,MacLachlan2014theoretical} also apply to the PCTL method, which implies that the PCTL algorithm is convergent. However, the convergence factor of the PCTL algorithm is still lack of a quantitative estimation, which is what we concern. 
According to the block structure of the linear system \cref{eq:main} and the PCTL algorithm, we introduce a specific format of the smoothing property and the approximation property, and derive an upper bound on the convergence factor of the PCTL algorithm based on these two properties, which provides theoretical guarantee for solving the 3-T problems by the PCTL algorithm. 
Furthermore, we also discuss the factors affecting the convergence of the PCTL algorithm from two perspectives. One is directly analyzing from the upper bound, which is sharp but usually expensive to compute. And the other is measuring the convergence factor by some easier-to-compute matrix properties, such as the diagonally dominant strength of $A_{\alpha}$ and coupling strength of $A$. It further provides an insight into the problems for which the PCTL algorithm is efficient.   

The rest of this paper is organized as follows. In \cref{sec:pctl}, we introduce the three-temperature equations and the PCTL algorithm for completeness. Then the convergence properties of the PCTL algorithm are analyzed in detail in \cref{sec:pctl_converge}. Moreover, we also discuss the factors affecting the convergence of the PCTL algorithm in \cref{sec:infulence_factor}.
Finally, some conclusions are given in \cref{sec:conclusion}.

\section{3-T energy equations and the PCTL algorithm}\label{sec:pctl}

In this section, we first briefly introduce the derivation of the 3-T linear systems \cref{eq:main}, and then present the PCTL algorithm.

\subsection{3-T energy equations}\label{ssec:3T}

Consider the 3-T energy equations \cite{Yue2015an,Xu2017algebraic}:

\begin{equation}\label{eq:3T}
  \begin{cases}
    \frac{\partial E_r}{\partial t} + \nabla \cdot (-\frac{c \lambda(E_r)}{\kappa_r} \nabla E_r) = c \kappa_{p} (E_p - E_r) \vspace{2mm} \\ 

   \rho c_e \frac{\partial T_e}{\partial t} + \nabla \cdot (- \kappa_e T_e^{5/2} \nabla T_e) = -c \kappa_{p} (E_p - E_r) + \omega_{ei}(T_i - T_e) \vspace{2mm} \\

   \rho c_i \frac{\partial T_i}{\partial t} + \nabla \cdot (- \kappa_i T_i^{5/2} \nabla T_i) = - \omega_{ei}(T_i - T_e) \\
  \end{cases}
\end{equation}
where $c$ is the speed of light, $\lambda(E_r)$ is a nonlinear limiter, $\rho$ is the medium density, $\omega_{ei}$ is the electron-ion coupling coefficient, $E_r$ and $E_p$ are the radiation and the electron scattering energy densities, respectively, $T_e$ and $T_i$ are the electron and ion temperatures, respectively, $\kappa_r$ and $\kappa_p$ are the Rosseland and the Planck mean absorption coefficients, respectively, $\kappa_e$ and $\kappa_i$ denote the diffusion coefficients of electron and ion, respectively, 
 $c_e$ and $c_i$ are the electron and ion heat capacity, respectively.
The equations \cref{eq:3T} describe the transform of radiation energy in the medium, as well as the energy exchange processes. Moreover, $E_p$ and $E_r$ can be defined as
$$
E_p = 4 \sigma T_e^4/c  \quad  E_r = 4 \sigma T_r^4/c.
$$

For the discretization of the 3-T equations, it usually uses fully implicit schemes, followed by the frozen-in coefficients method for linearization in the temporal direction, and numerous methods such as finite volume method in the spatial direction, which leads to 3-T linear system \cref{eq:main} to be solved.

\subsection{PCTL algorithm}\label{ssec:PCTL}

Realizing that the coupling relations of the 3-T linear systems makes the classical AMG algorithm inapplicable, Xu et al. \cite{Xu2009algebraic} proposed a specific coarsening strategy based on the structure and the properties of the 3-T linear systems. Combined with the C/F block relaxation, the PCTL algorithm decouples the fully coupled 3-T linear system \cref{eq:main} into some individual scalar subsystems that are easier to solve, which is described as in \cref{alg:pctl}. Details of the PCTL algorithm can also refer to \cite{Huang2022aSetup}.

\begin{algorithm}
\caption{PCTL algorithm for the linear system \cref{eq:main}}
\label{alg:pctl}
\begin{algorithmic}[1]
\REQUIRE{Matrix $A$: $\mathbb{R}^{3n \times 3n} \rightarrow \mathbb{R}^{3n \times 3n}$, 
right-hand side $b \in \mathbb{R}^{3n }$, initial guess 
$x^{(0)} := \begin{pmatrix}  {x_r^{(0)}}^T   &{x_i^{(0)}}^T  &{x_e^{(0)}}^T  \end{pmatrix}^T$,
 and the stop tolerance $\epsilon$.}
\ENSURE{Approximate solution $x$ fulfilling $\| b - Ax \|_2 / \| b \|_2 \leq \epsilon$.}
\STATE{{\bf Setup phase:} construct the interpolation operator $P = (P_r^T \, P_i^T \, I)^T$ and the restriction operator $R=P^T$, and then compute the associated coarse-level operator $A_c:=P^T A P$}
\STATE{ {\bf Solve phase: }while $\| b - Ax \|_2 / \| b \|_2 > \epsilon$ do
       \begin{itemize}
        \item[2.1] Pre-smoothing: do  C/F block smoothing
                \begin{equation*}
                    \begin{aligned}
                      x_r^{(k+1/3)} &= A_r^{-1} (b_r - D_{er}^T x_e^{(k)})\\
                      x_i^{(k+1/3)} &= A_i^{-1} (b_i - D_{ei}^T x_e^{(k)})\\
                      x_e^{(k+1/3)} &= A_e^{-1} (b_e - D_{er} x_r^{(k+1/3)}-D_{ei} x_i^{(k+1/3)})  
                  \end{aligned}
                \end{equation*} 
        \item[2.2] Coarse-grid solver: 
                  \begin{equation*}
                    A_c v_c = r_c = P^T (b-Ax^{(k+1/3)});
                  \end{equation*} 
       \item[2.3] Coarse-grid correction: 
                \begin{equation*}
                  \begin{aligned}
                    x_e^{(k+2/3)} &= x_e^{(k+1/3)} + v_c\\
                    x_r^{(k+2/3)} &= x_r^{(k+1/3)} + P_r v_c\\
                    x_i^{(k+2/3)} &= x_i^{(k+1/3)} + P_i v_c
                 \end{aligned}
                \end{equation*} 
       \item[2.4] Post-smoothing: do  F/C block smoothing
                 \begin{equation*}
                   \begin{aligned}
                    x_e^{(k+1)} &= A_e^{-1} (b_e - D_{er} x_r^{(k+2/3)}-D_{ei} x_i^{(k+2/3)})\\
                    x_r^{(k+1)} &= A_r^{-1} (b_r - D_{er}^T x_e^{(k+1)})\\
                    x_i^{(k+1)} &= A_i^{-1} (b_i - D_{ei}^T x_e^{(k+1)})
                    \end{aligned}
                  \end{equation*} 
       \end{itemize}
}
\end{algorithmic} 
\end{algorithm}


In the PCTL algorithm, the interpolation operator is selected as
$$
P = (P_r^T, P_i^T, I)^T.
$$
Note that in this case, the ideal interpolation operator is $P_{\alpha} = P_{\alpha}^{ex} := - A_{\alpha}^{-1}D_{\alpha e}$ $(\alpha = r,i)$, which is often dense and expensive to compute. Therefore, in order to save cost and to ensure the coarse-level operator $A_c:=P^T A P$ and the matrix $A_e$ have the same structure, the interpolation operators $P_{\alpha}$ in the PCTL algorithm are restricted to diagonal matrix. Moreover, it also satisfies
\begin{equation*}
  P_{\alpha} {\bf 1} = P_{\alpha}^{ex} {\bf 1} \quad (\alpha = r,i),
\end{equation*}
 where ${\bf 1} = (1,1,\cdots,1)^T \in \mathbb{R}^n$.

\section{Convergence of the PCTL algorithm}\label{sec:pctl_converge}

Throughout the paper we consider real matrices and adopt the following notations. Given a $n \times n$ matrix $A$, we use $\lambda(A)$, $\lambda_{min}(A)$, $\lambda_{max}(A)$ and $\rho(A)$ to represent the eigenvalues, the minimum eigenvalue, the maximum eigenvalue and the spectral radius of the matrix $A$.
Moreover, when the matrix $A$ is symmetric and positive definite, 
the A-norm or energy norm is defined by $\|x\|_A^2 = x^T A x$ with $x \in \mathbb{R}^n$, and the corresponding induced matrix norm is defined by $\|B\|_A = \max_{x \in \mathbb{R}^n, \|x\|_A = 1} \|Bx\|_A$. 

In this section, we aim to characterize the specific convergence properties of the PCTL algorithm, expecting to be helpful for further research, such as improving its efficiency and analyzing for which problems it is of high efficiency.
Before that, we first introduce some study on the convergence estimation of the ATG method.


\subsection{Convergence of the ATG method}\label{ssec:convTG}

The ATG method is composed by smoothing process and coarse-grid correction process. In general, the smoothing process works well at eliminating oscillatory errors and poorly at eliminating algebraically smooth errors, while the coarse-grid correction process follows to compensate it and to further reduce algebraically smooth errors, such that all errors could be quickly reduced. 

Consider the symmetric two-grid scheme, which is the simplest but the most representative scheme. Denote the pre- and post-smoother by $G_1:=I-M^{-T}A$ and $G_2:=I-M^{-1}A$, respectively, and denote the coarse-grid correction error propagator by $T:= I - P A_c^{-1} P^T A$, where $P$ is the interpolation operator and  $A_c:=P^T A P$ is the coarse-grid matrix constructed by the Galerkin strategy, then the error propagation matrix of the resulting two-grid method reads
\begin{equation*}
  E_{ATG} = G_2 T G_1 = (I-M^{-1}A)(I-P (P^T A P)^{-1} P^T A)(I-M^{-T}A).
\end{equation*}

The convergence theory of the ATG method mainly focuses on characterizing the (energy) norm of the error propagation matrix, that is $\|E_{ATG}\|_A$, and the study on it has been well developed \cite{Ruge1987algebraic,Stuben1999algebraic,Trottenberg2000multigrid,Falgout2005on,Notay2010algebraic}. In particular, the references \cite{McCormick1982multigrid,McCormick1985multigrid,Brandt1986algebraic,Mandel1988algebraic} have laid the foundation for numerous classical algebraic theoretical analysis of the ATG method. In recent decades, many universal convergence frameworks have been emerged for the exact or inexact AMG method and for symmetric or non-symmetric problems \cite{Falgout2004on,Falgout2005on,Notay2007convergence,Notay2010algebraic,Xu2022convergence,Xu2022a}, especially that the convergence factor $\|E_{ATG}\|_A$ of the two-grid method  has even be characterized by an elegant identity \cite{Falgout2005on}. 
However, the elegant identity is often impractical for its expensive computational cost, 
thus most researches turn to measuring the convergence rate of the ATG method by
finding a sharp upper bound on $\|E_{ATG}\|_A$ \cite{Ruge1987algebraic,Stuben1999algebraic,Trottenberg2000multigrid,MacLachlan2014theoretical}. 
A widely known strategy is translating the estimation into some sufficient conditions on the smoothing operator and coarse-grid correction operator, which are known as {\bf smoothing property} and {\bf weak approximation property} as described in \cref{lem:uniconverge1}. 

\begin{lemma} \label{lem:uniconverge1} \cite{MacLachlan2014theoretical}
 Assume $A$ is an SPD matrix, the interpolation operator $P$ is full rank, and post-smoother  $G_2=I - M^{-1} A$ is A-norm convergent. If there exist $\alpha_1, \beta_1>0$ independently of $e$ such that
          \begin{equation} \label{eq:post_converge}
             \begin{aligned}
                || G_2 e ||_A^2 & \leq ||e||_A^2 - \alpha_1 g(e)   
                \quad ({\rm post-smoothing \,\, property}) \\
                ||T e||_A^2 & \leq \beta_1 g(Te)    
                \quad ({\rm weak \,\, approximation \,\, property}),
             \end{aligned}
           \end{equation}
           where $g(e)$ is any non-negative function. Then the convergence factor of the ATG method satisfies
           \begin{equation}\label{eq:post_bound}
            \|G_2 T\|_A \leq \sqrt{1-\alpha_1/\beta_1}
           \end{equation}
           and thus
           \begin{equation}\label{eq:boundTG}
            \|E_{ATG}\|_A  = \| G_2T \|_A^2 \leq 1-\alpha_1/\beta_1.
           \end{equation}
\end{lemma}

Although this estimation leads to a certain loss of sharpness, it enables the effects of the smoothing process and coarse-grid correction process on the performance of the ATG algorithm more intuitive, 
and also  helps to make wise choices about  the components of an ATG algorithm for particular problems. 
In addition, it is worthy to note that different choices of the function $g(e)$ in \cref{eq:post_converge} will lead to different formats of the smoothing property and approximation property, and the readers can refer to \cite{MacLachlan2014theoretical} for more details.


\subsection{Convergence properties of the PCTL algorithm}

Although \cref{thm:tg_identity} states that the PCTL algorithm is convergent, a reasonable estimation of its convergence rate is  much significant and  beneficial, since it is helpful for evaluating the efficiency of the PCTL algorithm and for prejudging which problems the PCTL algorithm is of high efficiency. Thus in this subsection,  we choose a specific function $g(e)$ based on the block-structure of the PCTL algorithm and derive an upper bound for the convergence factor of the PCTL algorithm. 

First, we introduce some notations used in the following discussion,
\begin{equation*}
  \begin{aligned}
      A_{\alpha} &= (a_{kj}^{\alpha})_{n \times n}, \quad  A_{\alpha}^{-1} = (b_{kj}^{\alpha})_{n \times n}, \quad (\alpha=r,i,e)\\
     D_{er} &= {\rm diag} (d_{1}^r, \cdots, d_{n}^r), \quad D_{ei} = {\rm diag} (d_{1}^i, \cdots, d_{n}^i)\\
    P_{r} &= {\rm diag} (p_{1}^r, \cdots, p_{n}^r), \quad P_{i} = {\rm diag} (p_{1}^i, \cdots, p_{n}^i).
  \end{aligned}
\end{equation*}
Then the convergence factor of the PCTL algorithm is
\begin{equation*}
  ||E_{PCTL}||_A = ||((I-M^{-1}A)(I-P (P^T A P)^{-1} P^T A))||_A,
\end{equation*}
where the post-smoothing operator $G_2 = I-M^{-1}A$ is with
\begin{equation} \label{eq:smooth}
  M = \begin{bmatrix} 
    A_r &  0   & D_{er}^T \\
    0   & A_i  & D_{ei}^T \\
    0  & 0  & A_e 
  \end{bmatrix},
\end{equation}
and the interpolation operator $P = (P_r^T, P_i^T, I)^T$ is with
\begin{equation*}
  P_r = 
  \begin{bmatrix} 
    p_1^r &   &   \\ 
      & \ddots  & \\
      &  &  p_n^r 
  \end{bmatrix} = 
  \begin{bmatrix} 
    - \sum_{j=1}^n b_{1j}^r d_j^r &   &   \\ 
      & \ddots  & \\
      &  &  - \sum_{j=1}^n b_{nj}^r d_j^r 
  \end{bmatrix} 
\end{equation*}
and
\begin{equation*}
  P_i = 
  \begin{bmatrix} 
    p_1^i &   &   \\ 
      & \ddots  & \\
      &  &  p_n^i 
  \end{bmatrix}
  =
  \begin{bmatrix} 
    - \sum_{j=1}^n b_{1j}^i d_j^i &   &   \\ 
      & \ddots  & \\
      &  &  - \sum_{j=1}^n b_{nj}^i d_j^i 
  \end{bmatrix}.
\end{equation*}
 
Noticing that the smoothing operator and the interpolation operator in the PCTL algorithm are both of block structure, and recalling the symmetry and positive definitiveness of the matrices $A$, $A_r$, $A_i$ and $A_e$, we choose the function $g(e)$ in \cref{lem:uniconverge1}
as 
$$g(e):= \|e\|_{A \mathcal{D}^{-1} A}^2,$$ 
where 
\begin{equation}\label{eq:D}
  \mathcal{D} := 
\begin{bmatrix} 
  A_r &  0   & 0 \\
  0   & A_i  & 0 \\
  0 & 0  & A_e 
\end{bmatrix},
\end{equation}
then we get the following convergence conclusion.

\begin{theorem}\label{thm:converge1}
 For the PCTL algorithm, if there exist $\alpha, \beta>0$ independently of $e$ such that
          \begin{equation} \label{eq:post_converge1}
                || G_2 e ||_A^2 \leq ||e||_A^2 - \alpha  \|e\|_{A \mathcal{D}^{-1} A}^2 
          \end{equation}
          and
          \begin{equation} \label{eq:post_converge2}
                ||T e||_A^2 \leq \beta \|Te\|_{A \mathcal{D}^{-1} A}^2,
           \end{equation}
           where $\mathcal{D}$ is defined as in \cref{eq:D}, then the convergence factor of the PCTL algorithm satisfies
           \begin{equation}\label{eq:post_bound1}
            ||E_{PCTL}||_A \leq 1-\alpha/\beta.
           \end{equation}
\end{theorem}

It worth noting that the estimation in \cref{thm:converge1} consists with the special case $D_{er}=0$ and $D_{ei} =0$, that is there is no energy exchange between the three quantities. In this case, the PCTL algorithm degenerates into the direct method and thus $\|E_{PCTL}\|_A =0$.
On the other hand, from \cref{thm:converge1} (or from \cref{eq:smooth_property_2} and \cref{eq:approx_property_2}) it can be easily proved that $\alpha=1$ and $\beta=1$, which also yields $\|E_{PCTL}\|_A =0$. Whereas, taking the known choice $g(e) = \|e\|_{AD^{-1}A}$ ($D$ is the diagonal matrix of $A$) in \cref{lem:uniconverge1} can not leads to $\|E_{PCTL}\|_A =0$ in this case. That is one reason why we choose $g(e) = \|e\|_{A \mathcal{D}^{-1} A} $.
\begin{theorem}\label{thm:conv_pctl_special}
  Suppose there is no energy exchange between the three quantities, that is $D_{er}=0$ and $D_{ei} =0$, then the PCTL algorithm satisfies the post-smoothing property and the approximation property in \cref{thm:converge1} with
  $$
   \alpha = \beta =1,
  $$
  and then 
  $$
  \|E_{PCTL}\|_A =0.
  $$
\end{theorem}

Therefore, the \cref{thm:converge1} gives a feasible framework for assessing the convergence rate of the PCTL algorithm. 
In applications, there are some easier-to-compute and more intuitive substitutes for the smoothing and approximation properties, such as [\cite{Ruge1987algebraic}, Lemma 4.1 and Theorem 5.2] for $g(e) = \|e\|_{AD^{-1}A}$. Thus, we derive the similar substitutes for the post-smoothing property \cref{eq:post_converge1} and the weak approximation property \cref{eq:post_converge2}, which are described as in \cref{lem:smooth_property} and \cref{lem:uniconverge_tg}, respectively.                                          

\begin{lemma}\label{lem:smooth_property} 
  For the PCTL algorithm, if there exist $\alpha>0$ such that
  \begin{equation}\label{eq:smooth_property_2}
      \alpha M^T \mathcal{D}^{-1} M \leq M + M^T -A, 
  \end{equation}
 then the post-smoothing property \cref{eq:post_converge1} holds, where $A_1 \leq A_2$ represents the matrix $A_2-A_1$ is symmetric and positive semi-definite (SPSD).
\end{lemma}

\begin{lemma} \label{lem:uniconverge_tg} 
  For the PCTL algorithm, if there exist $\beta>0$ independent of $e$ such that 
  \begin{equation} \label{eq:approx_property_2}
    ||e - PSe||_{\mathcal{D}}^2 \leq \beta ||e||_A^2,  \quad \forall e
  \end{equation}
  then the weak approximation property \cref{eq:post_converge2} holds,
where $S=\begin{bmatrix} 0_n  & 0_n  & I_n \end{bmatrix}$, $0_n  \in \mathbb{R}^{n \times n}$ and $I_n  \in \mathbb{R}^{n \times n}$ are the matrices with all elements be zero and one, respectively.
\end{lemma}

In the following, we consider estimation of the convergence factor of the PCTL algorithm in the general case that the matrices $D_{er}$ and $D_{ei}$ are non-singular.
Before it, we firstly characterize some spectral properties of a matrix, which play an important role in our proof.

\begin{lemma} \label{lem:eig} \cite{Saad2003iterative}
  Suppose that matrices $A \in \mathbb{R}^{n \times n}$ and $B \in \mathbb{R}^{n \times n}$ are symmetric and positive definite, then there hold
  \begin{equation*}
    \lambda_{min}(A^{-1}B) \leq \frac{v^T B v}{v^T A v} \leq \lambda_{max}(A^{-1}B), \quad \forall v \in \mathbb{R}^{n}
  \end{equation*}
  and 
  $$
   \lambda_{min}(A^{-1}B) = \frac{1}{\lambda_{max}(B^{-1}A)}, \quad 
   \lambda_{max}(A^{-1}B) = \frac{1}{\lambda_{min}(B^{-1}A)}.
  $$
\end{lemma}

\begin{lemma} \label{lem:Gerschgorin} {({\bf Gerschgorin Disk Theorem})} \cite{Varga1962matrix}
  Let $A=(a_{ij})_{n \times n}$ be an arbitrary complex matrix, and let 
  $$
  \Lambda_i := \sum_{j=1, j \neq i}^{n} |a_{ij}|, \quad 1 \leq i \leq n,
  $$
  where $\Lambda_1:=0$ if $n=1$. If $\lambda$ is an eigenvalue of $A$, then there is a positive integer $r$, with $1 \leq r \leq n$, such that 
  $$
  |\lambda - a_{rr}| \leq \Lambda_r.
  $$
  Hence, all eigenvalues $\lambda$ of $A$ lie in the union of the disks
  $$
  |z - a_{ii}| \leq \Lambda_i, \quad 1 \leq i \leq n.
  $$
\end{lemma}

Now, we begin to discuss the convergence properties of the PCTL algorithm in detail. Specifically, we first show it satisfies the post-smoothing property  in \cref{thm:post_smooth_pctl}, and then prove it satisfies the approximation property in \cref{thm:approx_pctl}. Finally, we derive an estimation on the convergence factor of the PCTL algorithm in \cref{thm:converge_pctl}.

\begin{theorem}\label{thm:post_smooth_pctl}
  For the 3-T linear system \cref{eq:main}, the PCTL algorithm satisfies the post-smoothing property  $|| G_2 e ||_A^2 \leq ||e||_A^2 - \alpha  \|e\|_{A \mathcal{D}^{-1} A}$ with $\alpha$ be defined as
  \begin{equation} \label{eq:post_smooth_pctl}
    \alpha=\frac{\rho_s+2-\sqrt{\rho_s^2+4\rho_s}}{2},
  \end{equation}
  where $\rho_s$ is the spectral radius of the matrix $A_e^{-1}(D_{er}A_r^{-1}D_{er}^T + D_{ei}A_i^{-1}D_{ei}^T)$.
\end{theorem}

\begin{proof}
As stated in \cref{lem:smooth_property}, to prove the post-smoothing property of the PCTL algorithm, it is equivalent to prove that there exists a parameter $\alpha>0$ (as big as possible) such that 
  \begin{equation}\label{eq:post_smooth_pctl2}
    \alpha M^T \mathcal{D}^{-1} M \leq M+M^T-A,  
  \end{equation}
where
\begin{equation*}
  M+M^T-A =  
  \begin{bmatrix} 
    A_r & 0   & 0\\
    0  & A_i   & 0\\
    0  & 0 & A_e
 \end{bmatrix}
\end{equation*}
and 
\begin{equation*}
  M^T \mathcal{D}^{-1} M = 
\begin{bmatrix} 
  A_r & 0   &  D_{er}^T\\
 0  & A_i   & D_{ei}^T \\
 D_{er}  & D_{ei}   & A_e + D_{er}A_r^{-1}D_{er}^T + D_{ei}A_i^{-1}D_{ei}^T
\end{bmatrix}
,
\end{equation*}
which is further equivalent to that the matrix
\begin{equation}\label{eq:H}
  \begin{aligned}
    H &:= (M+M^T-A) - \alpha M^T \mathcal{D}^{-1} M \\
      &= \begin{bmatrix} 
        (1-\alpha) A_r & 0   &  -\alpha D_{er}^T\\
       0  & (1-\alpha) A_i   & -\alpha D_{ei}^T \\
       -\alpha D_{er}  & -\alpha D_{ei}   & (1-\alpha)A_e -\alpha (D_{er}A_r^{-1}D_{er}^T + D_{ei}A_i^{-1}D_{ei}^T)
      \end{bmatrix}
  \end{aligned}
\end{equation} 
is symmetric and positive semi-definite. It is easy to prove that $\alpha \leq 1$ and ``=" holds if and only if $D_{er} = D_{ei} =0$, which is discussed in \cref{thm:conv_pctl_special}. In the following, we only consider $\alpha < 1$, and thus the positive semi-definite property of $H$ just needs its Schur complement matrix
\begin{equation}\label{eq:H1}
  \begin{aligned}
    H_1 &:= (1-\alpha)A_e -\alpha (D_{er}A_r^{-1}D_{er}^T + 
            D_{ei} A_i^{-1}D_{ei}^T) - \frac{\alpha^2}{1-\alpha} (D_{er}A_r^{-1}D_{er}^T + D_{ei}A_i^{-1}D_{ei}^T) \\
      &= (1-\alpha)A_e - \frac{\alpha}{1-\alpha} (D_{er}A_r^{-1}D_{er}^T + D_{ei}A_i^{-1}D_{ei}^T)
  \end{aligned}
\end{equation} 
to be SPSD.

Note that the coefficient matrix $A$ is SPD, thus it can be checked that the matrix
$$
S_e:=A_e - D_{er}A_r^{-1}D_{er}^T - D_{ei}A_i^{-1}D_{ei}^T
$$
is also symmetric and positive definite.
Denote
$$
\hat{S}_e := A_e^{-1}(D_{er}A_r^{-1}D_{er}^T + D_{ei}A_i^{-1}D_{ei}^T)
$$
and 
$$\rho_s = \lambda_{max}(\hat{S}_e)$$
be the spectral radius of the matrix $\hat{S}_e$. Then according to the positive definitiveness of $S_e$ and the conclusions in Lemma \ref{lem:eig}, we could derive 
$$
\rho_s < 1
$$
and 
$$
 1-\alpha \geq \frac{\alpha}{1-\alpha} \rho_s.
$$
In conclusions, the PCTL algorithm satisfies the post-smoothing property \cref{eq:post_converge1} with
$$
\alpha = \frac{\rho_s+2-\sqrt{\rho_s^2+4\rho_s}}{2}.
$$
The proof is completed.
\end{proof}

Next, we prove the weak approximation property of the PCTL algorithm.

\begin{theorem}\label{thm:approx_pctl}
  For the 3-T linear system \cref{eq:main} with the coupling terms $D_{er}$ and $D_{ei}$ be non-singular, then the PCTL algorithm satisfies the following weak approximation property
  $
    ||e - PSe||_{\mathcal{D}}^2 \leq \beta ||e||_A^2
  $
  with 
  \begin{equation} \label{eq:Beta}
    \beta = \frac{(\rho_1 -2)\rho_s +1 }{1-\rho_s},
  \end{equation}
 where $\rho_s$ is the spectral radius of the matrix $A_e^{-1}(D_{er}A_r^{-1}D_{er}^T + D_{ei}A_i^{-1}D_{ei}^T)$ and 
 $$\rho_1 = \max \{ \frac{1}{\lambda_{min}^2(P_r^{-1}(-A_r^{-1}D_{er}))}, \, \frac{1}{\lambda_{min}^2(P_i^{-1}(-A_i^{-1}D_{ei}))} \}. $$
\end{theorem}
\begin{proof}

  In the PCTL algorithm, the two sides of the inequality  \cref{eq:approx_property_2} are 
  \begin{equation*}
    \|e\|_A^2 = 
    \begin{bmatrix} e_r\\ e_i \\e_e \end{bmatrix}^T 
    \begin{bmatrix} A_r & 0  & D_{er}^T  \\ 0 &A_i & D_{ei}^T \\
    D_{er}  &  D_{ei}  &A_e  \end{bmatrix} 
    \begin{bmatrix} e_r \\ e_i \\  e_e \end{bmatrix},
\end{equation*}
and
  \begin{equation*}
      ||e - PSe||_{\mathcal{D}}^2 = 
      \begin{bmatrix} e_r\\ e_i \\e_e \end{bmatrix}^T
      \begin{bmatrix} A_r  &0   &-A_r P_r\\  0  &A_i  & -A_i P_i \\
      -P_r A_r  & -P_i A_i  & P_r A_r P_r + P_i A_i P_i \end{bmatrix} 
      \begin{bmatrix} e_r\\ e_i \\e_e \end{bmatrix},
  \end{equation*}
  respectively.
Then, the weak approximation property $||e - PSe||_{\mathcal{D}}^2 \leq \beta ||e||_A^2$ is equivalent to that
\begin{equation*}
  \begin{aligned}
    Q := & \beta \begin{bmatrix} A_r & 0  & D_{er}^T  \\ 0 &A_i & D_{ei}^T \\
        D_{er}  &  D_{ei}  &A_e  \end{bmatrix} -
       \begin{bmatrix} A_r  &0   &-A_r P_r\\  0  &A_i  & -A_i P_i \\
         -P_r A_r  & -P_i A_i  & P_r A_r P_r + P_i A_i P_i \end{bmatrix}  \\
     =&  \begin{bmatrix} (\beta-1) A_r  & 0  & \beta D_{er}^T + A_r P_r \\
         0 & (\beta-1) A_i & \beta D_{ei}^T + A_i P_i\\
      \beta D_{er} + P_r A_r  & \beta D_{ei} + P_i A_i  & \beta A_e -(P_r A_r P_r + P_i A_i P_i) \end{bmatrix} 
  \end{aligned}
\end{equation*}
is a symmetric and positive semi-definite matrix. It is easy to prove that $\beta \geq 1$ and ``=" holds if and only if $D_{er} = D_{ei} =0$, which is discussed in \cref{thm:conv_pctl_special}. In the following, we only consider $\beta > 1$ and $D_{er}$ and $D_{ei}$ are non-singular, and then the positive semi-definite property of $Q$ just needs its block Schur complement matrix
\begin{equation}\label{eq:SQ1}
  \begin{aligned}
    S_Q :=& [\beta A_e - (P_r A_r P_r + P_i A_i P_i)] - \frac{1}{\beta-1} (\beta D_{er} + P_r  A_r)A_r^{-1} (\beta D_{er}^T + A_r P_r)  \\
    &- \frac{1}{\beta-1} (\beta D_{ei} + P_ i A_i)A_i^{-1} (\beta D_{ei}^T + A_i P_i) \\
      =& \beta A_e + \frac{2 \beta}{\beta-1} (P_r(-D_{er}) + P_i (-D_{ei})) - \frac{\beta}{\beta-1}(P_r A_r P_r + P_i A_i P_i) \\
      & - \frac{\beta^2}{\beta-1}(D_{er} A_r^{-1} D_{er} + D_{ei} A_i^{-1} D_{ei})
  \end{aligned}
\end{equation} 
to be SPSD.

For $\alpha=r,i$, denote
$$
L_{\alpha}:= (P_{\alpha})^{-1}(-A_{\alpha}^{-1}D_{e \alpha}).
$$
Taking $\alpha=r$ as an example, Lemma  \ref{lem:eig} shows that for any $x \in \mathbb{R}^n$, there holds 
\begin{equation}\label{eq:eigL1}
  \lambda_{min}(L_r) (x^T (-P_r D_{e r}) x) \leq x^T D_{er} A_{r}^{-1} D_{er} x \leq 
  \lambda_{max}(L_r) (x^T (-P_r D_{er}) x)
\end{equation}
and 
\begin{equation}\label{eq:eigL2}
  \lambda_{min}^2(L_r) (x^T P_r A_r P_r x) \leq x^T D_{er} A_r^{-1} D_{er} x \leq 
  \lambda_{max}^2(L_r) (x^T P_r A_r P_r x).
\end{equation}
On the other hand, the Gerschgorin Disk Theorem in \cref{lem:Gerschgorin} shows that for any eigenvalue $\lambda(L_r)$ of $L_r$, there holds
$$
 0 < \lambda(L_r) \leq 1.
$$
Moreover, the condition $P_{r} {\bf 1} = (-A_r^{-1}D_{er}) {\bf 1}$  further gives
$$
\lambda_{max}(L_r) =1.
$$
Thus, the inequalities \cref{eq:eigL1} and \cref{eq:eigL2} leads to
$$
x^T D_{er} A_{r}^{-1} D_{er} x \leq x^T (-P_r D_{er}) x
$$
and
$$
 x^T P_r A_r P_r x \leq \frac{1}{\lambda_{min}^2(L_r)} (x^T D_{er} A_r^{-1} D_{er} x),
$$
such that if we could prove that the matrix
$$
\begin{aligned}
  S :=& \beta A_e + \frac{2 \beta}{\beta-1} (D_{er} A_r^{-1} D_{er} + D_{ei} A_i^{-1} D_{ei}) - \frac{\beta}{(\beta-1) \lambda_{min}^2(L_{r})} D_{er} A_r^{-1} D_{er} - \\
  & \frac{\beta}{(\beta-1) \lambda_{min}^2(L_{i})} D_{ei} A_i^{-1} D_{ei}
     - \frac{\beta^2}{\beta-1}(D_{er} A_r^{-1} D_{er} + D_{ei} A_i^{-1} D_{ei})
\end{aligned}
$$
is SPSD, the positive semi-definite property of the matrix $S_Q$ is obtained.

Furthermore, denote
$$
\rho_1 = \max \{\frac{1}{\lambda_{min}^2(L_{r})}, \,  \frac{1}{\lambda_{min}^2(L_{i})} \} >1.
$$
Then if the matrix
$$
  S_1 := \beta A_e -\frac{\beta^2 + \rho_1 \beta - 2\beta}{\beta-1} (D_{er} A_r^{-1} D_{er} + D_{ei} A_i^{-1} D_{ei}) 
$$
is proved to be SPSD, the positive semi-definite property of the matrix $S$ follows. Recall that the matrix $\hat{S}_e = A_e^{-1}(D_{er}A_r^{-1}D_{er}^T + D_{ei}A_i^{-1}D_{ei}^T)$ is SPD and its spectral radius $\rho_s <1$, then the positive semi-definite property of $S_1$ is satisfied if
$$
\rho_s \leq \frac{\beta-1}{\beta + \rho_1 - 2},
$$
which further gives
$$
\beta \geq \frac{(\rho_1 -2)\rho_s +1 }{1-\rho_s}.
$$

Thus, the PCTL algorithm satisfies the weak approximation property \cref{eq:post_converge2} with
$$
\beta = \frac{(\rho_1 -2)\rho_s +1 }{1-\rho_s}.
$$
The proof is completed.
\end{proof}

Based on the \cref{thm:converge1}, an estimation for the convergence factor of the PCTL algorithm can be obtained.

\begin{theorem}\label{thm:converge_pctl}
  For the 3-T linear system \cref{eq:main} with the coupling terms $D_{er}$ and $D_{ei}$ be non-singular, the convergence factor of the PCTL algorithm satisfies 
  $$
  \|E_{PCTL}\|_A \leq \kappa:= \frac{\rho_s^2 + (2\rho_1 -3)\rho_s + (1-\rho_s) \sqrt{\rho_s^2 + 4 \rho_s}}{2(\rho_1-2)\rho_s +2},
  $$
  where $\rho_s$ and $\rho_1$ are defined as in \cref{thm:approx_pctl}.
\end{theorem}

\section{Analysis of the factors affecting the efficiency of PCTL algorithm} \label{sec:infulence_factor}

In this section, we discuss the factors affecting the convergence of the PCTL algorithm from two aspects. One is directly analyzing from the bound $\kappa$ and shows the impacts of the parameters $\rho_s$ and $\rho_1$. The other intends to study the influences of matrix properties on the convergence of PCTL algorithm, which gives a practical scheme to measure the convergence speed. It is expected that the discussions could provide guidances for further research on the PCTL algorithm, such as which kind of problems is the PCTL algorithm effective for and how to improve its efficiency. 

\subsection{Impacts of the parameters $\rho_s$ and $\rho_1$  on the efficiency of PCTL algorithm}

In order to visualize the influences of the parameters $\rho_s$ and $\rho_1$  on the convergence of the PCTL algorithm, we plot the curves of the upper bound $\kappa$ varies with the parameters $\rho_s$ and $\rho_1$ in Figure \ref{fig:convergence_rhos} and  Figure \ref{fig:convergence_rho1}, respectively. The figures shows that the convergence bound $\kappa$ increases both with the increase of the parameters $\rho_s$ and $\rho_1$. Moreover, it can be also observed that for small values of $\rho_s$, $\kappa$ does not vary much with $\rho_1$, while $\kappa$ varies greatly with the parameter $\rho_s$ even for small $\rho_1$.

\begin{figure}[!ht]
  \centering
    \includegraphics*[height=.7\textwidth,width=0.7\textwidth,keepaspectratio=true]{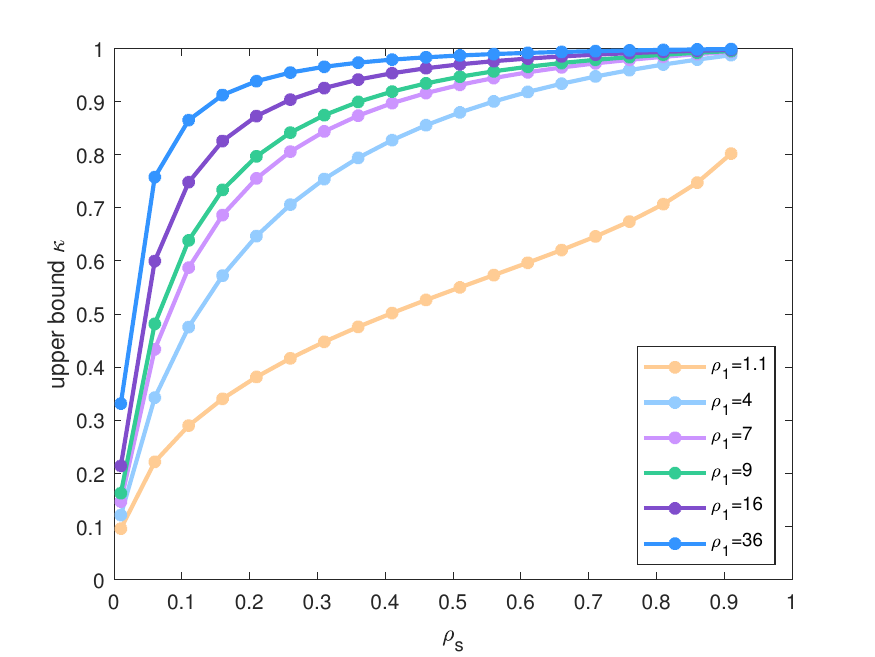}
    \caption{The convergence upper bound $\kappa$ of the PCTL algorithm varies with the parameter $\rho_s$.}\label{fig:convergence_rhos}
 \end{figure}

 \begin{figure}[!ht]
  \centering
    \includegraphics*[height=.7\textwidth,width=0.7\textwidth,keepaspectratio=true]{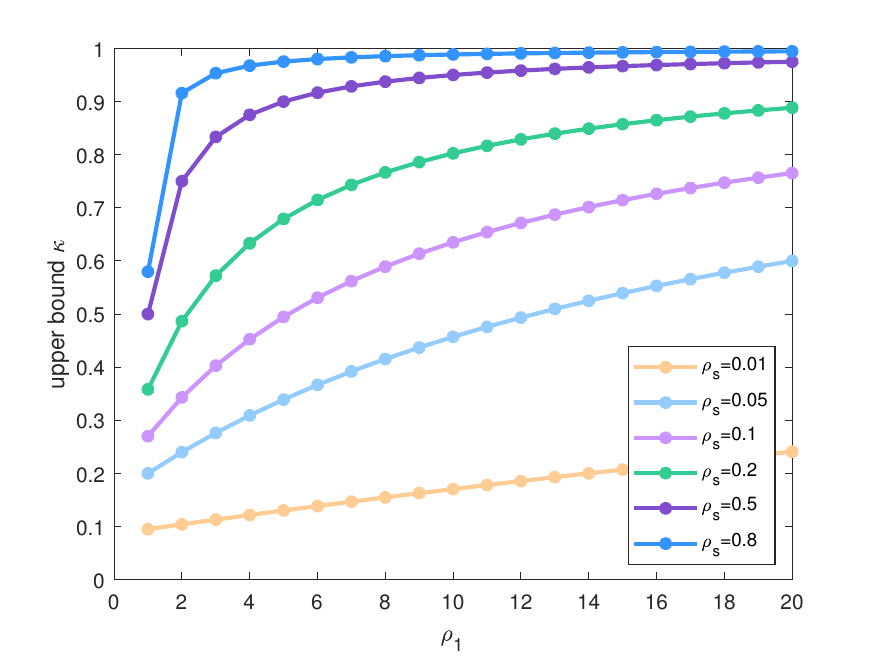}
    \caption{The convergence upper bound $\kappa$ of the PCTL algorithm varies with the parameter  $\rho_1$.}\label{fig:convergence_rho1}
 \end{figure}

Next, we try to give a theoretical explanation for the above phenomena.
\begin{itemize}
  \item[i)] Note that
        $$
        \rho_s = \lambda_{max}(A_e^{-1}(D_{er}A_r^{-1}D_{er}^T + D_{ei}A_i^{-1}D_{ei}^T)),
        $$
        and the matrix $S_e = A_e - (D_{er}A_r^{-1}D_{er}^T + D_{ei}A_i^{-1}D_{ei}^T)$ is the Schur-complement matrix of the coefficient matrix $A$, which reflects the strength of positive definitiveness of the linear system on coarse-level and reflects the strength of positive definitiveness of the linear system \eqref{eq:main} to some extent. Therefore, the smaller the parameter $\rho_s$, the stronger the positive definitiveness of coarse-grid matrix, and then the linear system on coarse-level is easier to solve, such that leads to the better convergence of the PCTL algorithm.
  \item[ii)] Since
            $$
              \frac{1}{\sqrt{\rho_1}} \leq \lambda(P_{\alpha}^{-1}(-A_{\alpha}^{-1}D_{e \alpha})) \leq 1,
            $$
            thus small value of $\rho_1$ indicates the interpolation operator $P_{\alpha}$ is a good approximation to the ideal interpolation operator $-A_{\alpha}^{-1} D_{e \alpha}$, which shows a logical fact that the better the interpolation operator approximates to the ideal interpolation, the better the efficiency of the PCTL algorithm. 
 \end{itemize}

 \subsection{Discussions of the influences of matrix properties on the efficiency of PCTL algorithm}

 In practical applications, the values of  $\rho_s$ and $\rho_1$ are usually not easy to calculate. Thus, in this subsection, we tend to give some easier-to-compute quantities to measure the convergence of the PCTL algorithm, including the diagonally dominant strength of $A_{\alpha}$ and coupling strength of $A$. To do it, we first introduce some quantities.
\begin{itemize}
  \item {\bf Diagonally dominant strength of $A_{\alpha}$} 
  
  Recall $A_{\alpha} = (a_{kj}^{\alpha})_{n \times n}$ $(\alpha=r,i,e)$, then define the diagonally dominant strength of the $k$-th ( $1 \leq k \leq n$)  row of $A_{\alpha}$ as
  \begin{equation}\label{eq:theta}
     \theta_k^{\alpha} :=  \frac{\sum_{1 \leq j \leq n} a_{kj}^{\alpha}}{a_{kk}^{\alpha}} \in (0,1].
  \end{equation}
  Note that $a_{kk}^{\alpha} >0$ and $a_{kj}^{\alpha} \leq 0 \, (j \neq k)$, such that the bigger $\theta_{k}^{\alpha}$ gives the stronger diagonal dominance.
  \item {\bf Coupling strength} 
  
  For the matrix $A$, define the coupling strength for the $k$-th ( $1 \leq k \leq n$) row as 
      \begin{equation}\label{eq:delta}
        \begin{aligned}
          & \delta_{k}^{r} := \frac{|d_{k}^{r}|}{a_{kk}^{r}} \in [0,\theta_{k}^{r}), \quad 
          \delta_{k}^{i} := \frac{|d_{k}^{i}|}{a_{kk}^{i}} \in [0,\theta_{k}^{i}), \\
          & \delta_{k}^{e} := \frac{|d_{k}^{r} + d_{k}^{i}|}{a_{kk}^e} \in [0,\theta_{k}^{e}).
        \end{aligned}
      \end{equation}
  Obviously, the bigger $\delta_k^{\alpha}$ indicates the stronger coupling.
\end{itemize}

\begin{theorem}\label{thm:convergence_pctl_loose}
  For the 3-T linear system \cref{eq:main} with the coupling terms $D_{er}$ and $D_{ei}$ be non-singular, the convergence factor of the PCTL algorithm satisfies
  $$
  \|E_{PCTL}\|_A \leq \frac{\mu_s^2 + (2 \mu_1^2 -3)\mu_s + (1-\mu_s) \sqrt{\mu_s^2 + 4 \mu_s}}{2(\mu_1^2 -2)\mu_s +2},
  $$
  with 
  $$
  \mu_1 = \max_{1 \leq k \leq n} \{ \frac{(2-\theta_k^{r})(1-\theta_k^r + \delta_k^r)}{\delta_k^r},\,\frac{(2-\theta_k^{i})(1-\theta_k^i + \delta_k^i)}{\delta_k^i} \}
 $$
 and 
 $$
 \mu_s = \max_{1 \leq k \leq n} \{ \frac{\delta_k^e}{\theta_k^e} \},
 $$
 where $\theta_k^{\alpha}$ and $\delta_k^{\alpha}$ are defined as in \cref{eq:theta} and \cref{eq:delta}.
\end{theorem}

\begin{proof}

From the proof in \cref{thm:post_smooth_pctl} and \cref{thm:approx_pctl}, one could derive that if there exists parameters $\mu_1$ and $\mu_s$ such that 
$$
 \rho_s \leq \mu_s \quad \text{ and }  \quad  \rho_1 \leq \mu_1^2, 
$$
then the post-smoothing and approximation properties of the PCTL algorithm also hold with $\rho_s$ and $\rho_1$ in \cref{thm:post_smooth_pctl} and \cref{thm:approx_pctl} be replaced by $\mu_s$ and $\mu_1^2$, respectively, and thus lead to  
$$
\|E_{PCTL}\|_A \leq \frac{\mu_s^2 + (2 \mu_1^2 -3)\mu_s + (1-\mu_s) \sqrt{\mu_s^2 + 4 \mu_s}}{2(\mu_1^2 -2)\mu_s +2}.
$$

On the one hand, according to \cref{lem:eig}, $\rho_1 \leq \mu_1^2$ is equivalent to that for $\forall x \in \mathbb{R}^n$, there hold
$$
\frac{x^T(-D_{er} P_r^{-1})x}{x^T A_r x} \geq \frac{1}{\mu_1} \quad \text{ and }  \quad     \frac{x^T(-D_{ei} P_i^{-1})x}{x^T A_i x} \geq \frac{1}{\mu_1},
$$
which further indicate 
$$
-\mu_1 D_{er} P_r^{-1} - A_r \geq 0 \quad \text{ and }  \quad -\mu_1 D_{ei} P_i^{-1} - A_i \geq 0.
$$ 
 From the
Gerschgorin Disk Theorem in \cref{lem:Gerschgorin} we know, if a symmetric matrix is diagonal dominant with positive diagonal elements, it must be a SPSD matrix. Thus, if
there holds
$$
\frac{-\mu_1 d_k^{\alpha}}{p_k^{\alpha}} - a_{kk}^{\alpha} \geq \sum_{j\neq k}(-a_{kj}^{\alpha}) \quad 1\leq k \leq n, \, \alpha=r,i,
$$
that is 
$$
\mu_1 = \max_{1 \leq k \leq n} \{ \frac{(2-\theta_k^{r}) p_k^r}{\delta_k^r},\,\frac{(2-\theta_k^{i}) p_k^i}{\delta_k^i} \},
$$
the matrices $-\mu_1 D_{er} P_r^{-1} - A_r$ and $-\mu_1 D_{ei} P_i^{-1} - A_i$ must be SPSD. 

Next, we discuss the the range of $p_k^{\alpha}$ and take $p_1^r$ as an example. The property $A_r^{-1} A_r =I$ gives $\sum_{j=1}^n b_{1j}^r a_{j1}^r =1$ and
\begin{equation*}
  \begin{aligned}
    & \sum_{l=1}^n b_{1l}^r a_{lj}^r = 0, \quad 2 \leq j \leq n \\
    \Rightarrow  \ &b_{11}^r \sum_{l=2}^n a_{1l}^r  + \cdots +  b_{1n}^r \sum_{l=2}^n a_{nl}^r = 0\\
    \Rightarrow \  &b_{11}^r (s_1 - a_{11}^r - d_1^r) + \cdots + b_{1n}^r (s_n - a_{n1}^r - d_n^r) =0,
  \end{aligned}
\end{equation*}
which further lead to
\begin{equation}\label{eq:Zero_r2}
  1 - p_1^r = \sum_{j=1}^n b_{1j}^r s_j >0,
\end{equation} 
where $s_j$ denotes the sum of the $j$-th row's elements of the matrix $A$.
Furthermore, $\sum_{j=1}^n b_{1j}^r a_{j1}^r =1$ also shows that 
$$
b_{11}^r \geq \frac{1}{a_{11}^r},
$$
then it gives
$$
1-p_1^r \geq b_{11}^r s_1 \geq \frac{s_1}{a_{11}^r} = \theta_1^r - \delta_1^r \hspace{2mm} \Rightarrow
\hspace{2mm}  p_1^r \leq 1-\theta_1^r + \delta_1^r.
$$
In the same way, we can obtain
$$
p_k^{\alpha} \leq 1-\theta_k^{\alpha} + \delta_k^{\alpha}, \quad 1\leq k \leq n, \, \alpha=r,i.
$$
Thus $\mu_1$ could also take 
$$
\mu_1 = \max_{1 \leq k \leq n} \{ \frac{(2-\theta_k^{r})(1-\theta_k^r + \delta_k^r)}{\delta_k^r},\,\frac{(2-\theta_k^{i})(1-\theta_k^i + \delta_k^i)}{\delta_k^i} \},
$$
Although it leads to a looser upper bound on the convergence of the PCTL algorithm, it is practical and very cheap to compute.

On the other hand,
$$
\frac{x^T(D_{er}A_r^{-1}D_{er}^T + D_{ei}A_{i}^{-1}D_{ei}^T)x}{x^T A_e x} \leq \rho_s \leq \mu_s, \quad \forall x \in \mathbb{R}^n
$$
is equivalent to the matrix 
$$
\mu_s A_e - (D_{er}A_r^{-1}D_{er}^T + D_{ei}A_{i}^{-1}D_{ei}^T)
$$
be SPSD, which is further equivalent to the matrix
$$
\hat{A}:=
\begin{bmatrix} 
  A_r &  0   & D_{er}^T \\
  0   & A_i  & D_{ei}^T \\
  D_{er} & D_{ei}  & \mu_s A_e 
\end{bmatrix}
$$
be a SPSD matrix. According to the Gerschgorin Disk Theorem in \cref{lem:Gerschgorin}, the positive semi-definitiveness of the matrix $\hat{A}$ can be satisfied if 
\begin{equation}\label{eq:mus}
  \mu_s a_{kk}^e \geq \mu_s \sum_{j=2}^n |a_{kj}^e| + |d_k^r| +|d_k^i|, \quad 1 \leq k \leq n,
\end{equation}
where the diagonal dominance of the first $2n$ rows of $\hat{A}$ is guaranteed by the diagonal dominance of the coefficient matrix $A$.
Note that $a_{kj}^e \leq 0 (k \neq j)$ and $d_k^r, d_k^i < 0$, then the inequality \cref{eq:mus} leads to
$$
\mu_s = \max_{1 \leq k \leq n} \{ \frac{\delta_k^e}{\theta_k^e} \}.
$$
Thus, the conclusions of Theorem \ref{thm:convergence_pctl_loose} are obtained. The proof is completed.
\end{proof}

In order to further visualize the impacts of the diagonally dominant strength and the coupling strength on the convergence of the PCTL algorithm, we plot the curves that the convergence bound $\kappa_1$ varies with these two factors as shown in Figure \ref{fig:convergence_theta} and Figure \ref{fig:convergence_delta}, respectively. Moreover, to be able to show in greater details the influences of these two factors,  we set the diagonally dominant strength and the coupling strength of all rows are the same, and denote the diagonally dominant strength and the coupling strength by $\delta$ and $\theta$, respectively. The figures illustrate  that the convergence factor is negatively correlated to the diagonally dominant strength and positively related to the coupling strength as a whole, which show the stronger the diagonal dominance of $A_{\alpha}$ is and the weaker of the coupling terms are, the better the convergence of the PCTL algorithm is. 

\begin{figure}[!ht]
  \centering
    \includegraphics*[height=.7\textwidth,width=0.7\textwidth,keepaspectratio=true]{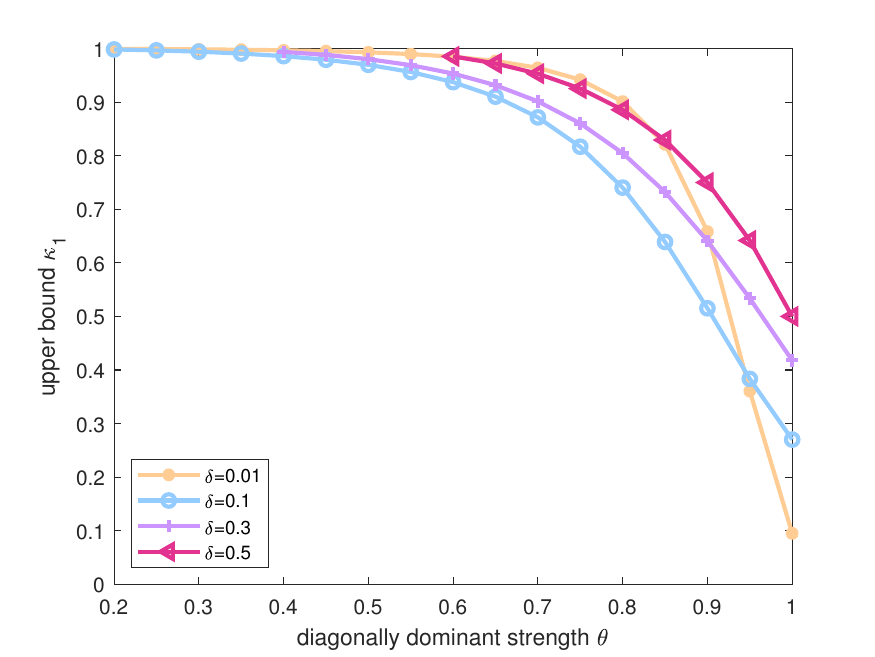}
    \caption{The convergence upper bound of the PCTL algorithm varies with the diagonally dominant strength.}\label{fig:convergence_theta}
 \end{figure}

 \begin{figure}[!ht]
  \centering
    \includegraphics*[height=.7\textwidth,width=0.7\textwidth,keepaspectratio=true]{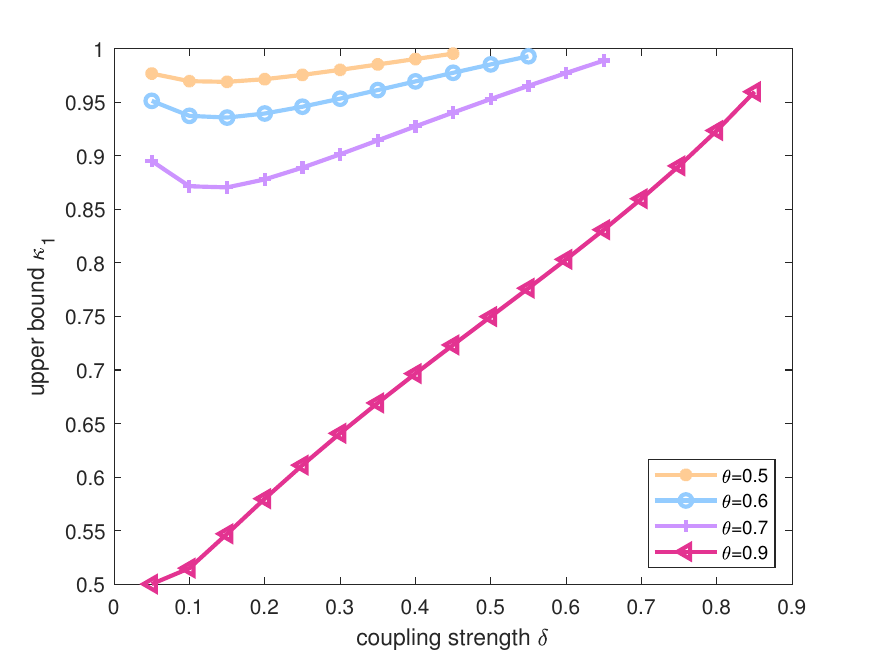}
    \caption{The convergence upper bound $\kappa$ of the PCTL algorithm varies with the coupling strength.}\label{fig:convergence_delta}
 \end{figure}

\section{Conclusions}\label{sec:conclusion}

In this work, the convergence properties of the PCTL algorithm have been analyzed. Specifically, by proving the PCTL algorithm satisfies the smoothing property and approximation property, we have derived an upper bound on its convergence factor.  Moreover, we have also discussed the factors that affect the convergence of the PCTL algorithm and given some easier-to-compute quantities (diagonally dominant strength and the coupling strength) to measure it. It shows that the more diagonally dominant the matrix $A_{\alpha}$ is, the
better the efficiency of the PCTL algorithm is.

\section*{Acknowledgments}

The work is financially supported by the China Postdoctoral Science Foundation (2022M710461) and the National Natural Science Foundation of China
(62032023).

\bibliographystyle{siamplain}
\bibliography{references}
\end{document}